\documentclass[12pt]{article}
\usepackage{amsfonts, amssymb}
\usepackage{amsmath}
\usepackage{amsthm}

\newtheorem{theorem}{Theorem}[section]

\title{ On the recurrence coefficients of the generalized  little $q$-Laguerre polynomials}

\author{Galina    Filipuk\footnote{Faculty of Mathematics, Informatics and Mechanics, University of Warsaw, Banacha 2, Warsaw 02-097, Poland. Email: filipuk@mimuw.edu.pl}\ \  and Christophe   Smet\footnote{Department of Mathematics, KU Leuven, Celestijnenlaan 200 B box 2400, BE-3001, Leuven,
Belgium. Email: chr.smet@gmail.com} }
\date{}

\begin{document}
\maketitle

\begin{abstract}
In this paper we consider a semi-classical variation of the weight related to the little  q-Laguerre polynomials and obtain a second order second degree discrete equation for the recurrence coefficients in the three-term recurrence relation.
\end{abstract}

 

\section{Introduction}

\subsection{Orthogonal polynomials}

Orthogonal polynomials appear in many areas of modern mathematics and mathematical physics \cite{Chihara, Ismail} (e.g., approximation theory, stochastic processes, random matrix theory and others). In this paper we are interested in discrete $q$-orthogonal polynomials on an exponential lattice. The orthogonality condition for discrete $q$-orthonormal polynomials is given by
$$
\int_a^b p_k(x) p_n(x) w(x) d_q x= \delta_{k,n},
$$
where the $q$-integral  \cite{GR} is defined by
$$
\int_a^b f(x) d_q x=b(1-q)\sum_{n=0}^{\infty} q^n f(b q^n)-a(1-q)\sum_{n=0}^{\infty}q^n f(a q^n).
$$
Here the weight function $w$ is  supported on the exponential lattice $$\{a q^n,\;b q^n\,|\,n\in\mathbb{N}_0\}$$ and $\delta_{k,n}$ is the Kronecker delta.

The classical examples include little $q$-Laguerre polynomials, which are orthogonal on the exponential lattice $\{q^k\,|\,k\in\mathbb{N}_0\}$ with respect to the weight function $$w(x)=x^{\alpha}(qx;q)_{\infty},\;\;\alpha>-1,\;\;q\in(0,1),$$ where $$(a;q)_{\infty}=\prod_{k=0}^{\infty}(1-a q^k).$$ They can be written in terms of the basic hypergeometric function ${}_2\phi_1$.

One of the main features of orthogonal polynomials is the three-term recurrence relation $$x p_n(x)=a_{n+1}p_{n+1}(x)+b_np_n(x)+a_n p_{n-1}(x).$$ Here $\{p_n(x)\}$ are orthonormal polynomials of degree $n$ and the coefficients $a_n$ and $b_n$ are usually referred to as the recurrence coefficients. They possess a number of important properties. For instance, they can be expressed in terms of determinants containing the moments of the orthogonality measure \cite{Chihara}. Moreover, for classical orthogonal polynomials they are known explicitly. For other, non-classical, polynomials the recurrence coefficients are not known explicitly and sometimes they can be expressed in terms of the solutions of discrete (including $q$-discrete)  or continuous Painlev\'e equations. The Painlev\'e equations in their turn have many remarkable applications in modern  mathematics and mathematical physics (see for instance \cite{Clarkson} and the references therein). There are a few examples of relations of $q$-orthogonal polynomials on an exponential lattice to the $q$-discrete Painlev\'e equations: the weight $$w(x)=(q^4x^4;q^4)_{\infty}$$ on $\{\pm q^k\,|\,k\in\mathbb{N}_0\}$ and ${q\textrm{P}_\textrm{\small{I}}}$ \cite{VanAssche}; the weight $$w(x)=|x|^{\alpha}(q^2x^2;q^2)_{\infty}(cq^2x^2;q^2)_{\infty}$$ on $\{\pm q^k \,|\,k\in\mathbb{N}_0\}$ and ${\alpha q\textrm{P}_\textrm{\small{V}}}$ \cite{Lies, VanAssche}; the weight $$w(x)=x^{\alpha}(q^2x^2;q^2)_{\infty}$$ on $\{ q^k \,|\,k\in\mathbb{N}_0\}$ and ${q\textrm{P}_\textrm{\small{V}}}$ \cite{Lies}. For other examples of relations of the recurrence coefficients for the orthogonal polynomials, not necessarily supported on an exponential lattice, see for instance \cite{BFV, BV, GF+Walter2, Freud, Magnus, Magnus2} and the references therein. One of the methods to derive the nonlinear discrete equations for the recurrence coefficients is by using the ladder operators.

\subsection{Ladder operators}
In the case of discrete $q$-orthogonal polynomials on the exponential lattice, the ladder operators were first
considered in \cite{Ismailq}. %The general theory of ladder operators by using the times scales calculus is considered in \cite{FilipukHilger}.
 We repeat the main statements which we use later on to be self-contained following \cite{Ismailq} and    \cite[Section~1.3]{Lies}.

The $q$-difference operator is given by
$$
(D_q f)(x)=\begin{cases}
\frac{f(x)-f(qx)}{x(1-q)},& x\neq 0,\\
f'(0)&x=0.
\end{cases}
$$

Consider a weight function $w$ on the exponential lattice $\{a q^n,\;b q^n\,|\,n\in\mathbb{N}_0\},$ such that $w(a/q)=w(b/q)=0$ and the sequence of orthonormal polynomials $\{p_n\}$ of degree $n$ with respect to this weight. Ismail \cite{Ismailq} shows that the polynomials satisfy the following relation:
$$D_q p_n(x)=A_n(x)p_{n-1}(x)-B_n(x)p_n(x)$$
with
\begin{equation}\label{An}
A_n(x)=a_n\int_a^{b}\frac{u(qx)-u(y)}{qx-y}p_n(y)p_n(y/q)w(y)d_q y,
\end{equation}
\begin{equation}\label{Bn}
B_n(x)=a_n\int_a^{b}\frac{u(qx)-u(y)}{qx-y}p_n(y)p_{n-1}
(y/q)w(y)d_q y.
\end{equation}
Here the function $u$, called the potential, is defined by the following formula:
\begin{equation}\label{potential}
-u(qx)w(qx)=D_{q}w(x).
\end{equation}
Furthermore, the following relations (compatibility conditions) hold:
\begin{equation}\label{comp1}
B_{n}+B_{n+1}=(x-b_n)\frac{A_n}{a_n}+(q-1)x\,\sum_{j=0}^n \frac{A_j}{a_j}-u(qx),
\end{equation}
\begin{equation}\label{comp2}
a_{n+1}A_{n+1}-a_n^2\frac{A_{n-1}}{a_{n-1}}=(x-b_n)B_{n+1}-(qx-b_n)B_n+1.
\end{equation}
Relations (\ref{comp1}), (\ref{comp2}) are important in deriving nonlinear discrete equations for the recurrence coefficients, which in some cases can be further reduced to ($q$-)discrete Painlev\'e equations.

\section{Main results}

In this paper we study the recurrence coefficients for the weight functions supported on the exponential lattice $\{ q^k \,|\,k\in\mathbb{N}_0\}$ and satisfying the $q$-difference equation (\ref{potential}) with
\begin{equation}\label{pot gen}
u(x)=\frac{k_1 q}{1-q}\frac{1}{x}+\frac{k_2 x+k_3}{1-q},\;\;k_1\neq 0,\;\;k_2\neq 0,
\end{equation}
and conditions $w(0)=w(1/q)=0$. At the end of this section we discuss the existence of such weight functions and consider a few instructive examples. In the following we assume that
the sequence of polynomials $\{p_n\}$
is orthonormal with respect to the weight function with potential (\ref{pot gen}) and hence the orthogonality relation takes the form
$$\int_0^1 p_m(x)p_n(x)w(x)d_q x=\delta_{m,n}.$$

It is straightforward to calculate that
$$\frac{u(qx)-u(y)}{qx-y}=\frac{k_1}{(q-1)xy}+\frac{k_2}{1-q}.$$
Hence,  the expressions (\ref{An}) and (\ref{Bn}) can be computed as follows:
$$A_n(x)=\frac{a_n R_n}{x(1-q)}+\frac{a_n k_2 q^{-n}}{1-q}, $$
$$B_n(x)=\frac{r_n}{(1-q)x},$$
where $$R_n=-k_1\int_0^1  p_n(y)p_{n}(y/q)\frac{w(y)}{y}d_q y,\;\;r_n=-a_n k_1\int_0^1 p_n(y)p_{n-1}(y/q)\frac{w(y)}{y}d_q y$$
and we have used orthogonality in computing $$\int_0^1  p_n(y)p_{n}(y/q)w(y)d_q y =q^{-n},\;\;\int_0^1  p_n(y)p_{n-1}(y/q)w(y)d_q y =0.$$

The compatibility conditions (\ref{comp1}), (\ref{comp2}) give rise to the following system (after comparing the coefficients at the powers of $x$):
\begin{equation}\label{eq733g}
r_{n+1}+r_n=-b_n R_n-k_1,
\end{equation}
\begin{equation}\label{eq734g}
R_n-k_2 q^{-n}b_n-k_3-(1-q)\sum_{j=0}^nR_j=0,
\end{equation}
\begin{equation}\label{eq735g}
a_{n+1}^2R_{n+1}-a_n^2 R_{n-1}=-b_n(r_{n+1}-r_n),
\end{equation}
\begin{equation}\label{eq736g}
k_2 a_{n+1}^2-k_2 q^2 a_n^2=q^{1+n}-q^{2+n}-q^{2+n}r_n+q^{1+n}r_{n+1}.
\end{equation}

We will use these equations to find expressions for the recurrence
coefficients $a_n$, $b_n$ of the sequence of orthonormal polynomials
$\{p_n\}$  with respect to $w$. Multiplying (\ref{eq736g}) by
$q^{-2n-2}$ and taking a telescopic sum with $a_0=r_0=0$, we get
\begin{equation}\label{eq737g}
a_n^2= \frac{q^n(1-q^n+r_n)}{k_2}.
\end{equation}

Multiplying (\ref{eq735g}) by $R_n$, substituting $-b_n R_n$ from (\ref{eq733g}) and taking a telescopic sum, we obtain
\begin{equation}\label{eq738g}
a_n^2 R_n R_{n-1}=r_n(k_1 +r_n).
\end{equation}

On the other hand, if we substitute the expression for $b_n$ from
(\ref{eq734g}) and the expression for $a_n^2$ from (\ref{eq737g})
into (\ref{eq735g}) and we collect the terms in $r_n$ and $r_{n+1}$,
we obtain
\begin{multline*}r_{n+1}\left(qR_{n+1}+R_n-k_3-(1-q)\sum_{j=0}^nR_j\right)-r_n\left(qR_n+R_{n-1}-k_3-(1-q)\sum_{j=0}^{n-1}R_j\right)\\
=(1-q^n)R_{n-1}-q\left(1-q^{n+1}\right)R_{n+1}.\end{multline*} The
left hand side of this expression can easily be summed
telescopically and in the resulting equation we recognize the
expression for $b_n$ from (\ref{eq734g}), so we get
$$r_n\left(R_{n-1}+k_2b_nq^{-n}\right)=q^{n+1}R_n+q^nR_{n-1}-R_{n-1}-k_2b_nq^{-n}-k_3.$$
If we multiply this by $R_n$, we can use expressions (\ref{eq733g})
to substitute $b_nR_n$ and (\ref{eq737g}), (\ref{eq738g}) to
substitute $R_nR_{n-1}$.  Eventually we get a quadratic equation for
$R_n$:
\begin{equation}\label{eq7310g}
R_n^2-k_3q^{-n-1}R_n=-k_2q^{-2n-1}\left((1+r_n)(1+r_{n+1})-(1-k_1)\right).
\end{equation}

From (\ref{eq737g})  and (\ref{eq738g}) we get
\begin{equation}\label{eq738gnew}
q^{n}(1-q^n+r_n) R_n R_{n-1}=k_2r_n(r_n+k_1).
\end{equation}

We can use equations (\ref{eq7310g}) and (\ref{eq738gnew}) to get a second order second degree difference equation for $r_n$.
  In order to derive such an equation (which we omit here explicitly as it is long and cumbersome) we can first eliminate $R_n$ between equations (\ref{eq7310g})
 and (\ref{eq738gnew}) and then eliminate $R_{n-1}$ between the obtained equation and (\ref{eq7310g}) with $n$ replaced by $n-1$.

Alternatively, we can first   replace equation
(\ref{eq734g}) by subtracting from it (\ref{eq734g}) with $n-1$ to
eliminate the sum of $R_n$. We have the following equation instead
of (\ref{eq734g}):
\begin{equation}\label{eq734gnew}
k_2 q b_{n-1}-k_2 b_n+q^n(q R_n-R_{n-1})=0.
\end{equation}
 From (\ref{eq733g}) we can find $b_n$ in terms of $r_n$ and $R_n$ and substitute this expression into equations (\ref{eq734g})--(\ref{eq736g}).
 Further, we use (\ref{eq737g}).
Substituting (\ref{eq737g}) into (\ref{eq738g}), we can find $R_{n-1}$ and, by taking (\ref{eq737g}) and (\ref{eq738g}) with $n$ replaced by $n+1$,
we can find $R_{n+1}$ in terms of $R_n$ and $r_n$ ($r_{n+1}$ respectively). We have \begin{equation}\label{R{n-1}g}
R_{n-1}=-\frac{k_2q^{-n}r_n(r_n+k_1)}{(q^n-1-r_n)R_n},
\end{equation}
\begin{equation}\label{R{n+1}g}
R_{n+1}=-\frac{k_2 q^{-1-n}r_{n+1}(r_{n+1}+k_1)}{(q^{n+1}-1-r_{n+1})R_n}.
\end{equation}
Substituting these expressions into (\ref{eq734gnew}), we can find
$R_n^2$:
\begin{equation}\label{Rn2g}
R_n^2=-\frac{k_2q^{-1-n}r_n(r_n+k_1)(r_n(q^n-1-r_{n+1})+(q^n-1)(r_{n+1}+k_1))}
{(q^n-1-r_n)(r_{n-1}(q^n-1-r_n)+(q^n-1)(r_n+k_1))}.
\end{equation}

From (\ref{eq737g}) we have an expression of $a_n^2$ in terms of
$r_n.$ Using (\ref{Rn2g}) and (\ref{eq7310g}) we can get an
expression of $R_n$ in terms of $r_{n-1},\;r_n$ and $r_{n+1}$.
Hence, using (\ref{eq737g}) we can get an expression of $b_n$ in
terms of $r_{n-1},\;r_n$ and $r_{n+1}$. Thus, we have proved the
following theorem.

\begin{theorem}
The recurrence coefficients $a_n$, $b_n$ appearing in the three-term recurrence relation for the weight supported on the exponential lattice $\{ q^k \,|\,k\in\mathbb{N}_0\}$ with the potential   satisfying (\ref{pot gen}) and  conditions $w(0)=w(1/q)=0$ can be expressed in terms
of $r_n$ (by using (\ref{eq737g}), (\ref{eq733g}) with (\ref{Rn2g})), which is a solution of  a second order second degree discrete equation (with respect to $n$) given by
\begin{equation}\label{eq:thm1}f^2-2f g+g^2-k_3^2 q^{-2n-2}f=0,\end{equation} where $f$ denotes the right hand side of (\ref{Rn2g}) and $g$ is the right hand side of  (\ref{eq7310g}).
\end{theorem}

Note that the discrete equation in the theorem  factorizes if and only if $k_3=0$ and this allows us to express $r_{n+1}$ in terms of $r_{n-1}$ and $r_n$ as a rational function. The connection is given by a $q$-Painlev\'{e} equation:

\begin{theorem}\label{thm:qpV}
If $k_3=0$ in (\ref{pot gen}) then the variable $x_n=(1+r_n)(1-k_1)^{-1/2}$ satisfies $qP_V$ \cite{Ramani} given by
\begin{equation}\label{qP5}
(x_n x_{n-1}-1)(x_n x_{n+1}-1)=
\frac{\gamma\delta
q^{2n}(x_n-\alpha)(x_n-1/\alpha)(x_n-\beta)(x_n-1/\beta)}{(x_n-\gamma
q^n)(x_n-\delta q^n)},
 \end{equation}
with \[\alpha=\beta=\gamma=\delta=\frac{1}{p}\]  where $p=\sqrt{1-k_1}$.
The initial conditions are given by \[x_0=\frac{1}{p}\qquad{\rm and}\qquad x_1=p-\frac{k_2}{qp}\left(\frac{\mu_1}{\mu_0}\right)^2.\]
$x_n$ is related to the recurrence coefficients $a_n$ and $b_n$ of the orthogonal polynomials by \begin{equation}\label{aninThm2}a_n^2=\frac{q^n}{k_2}\left(px_n-q^n\right)\end{equation} and \begin{equation}\label{bninThm2}b_n^2=-\frac{q^{2n+1}(px_n+px_{n+1}-1-p^2)^2}{k_2p^2(x_nx_{n+1}-1)}.\end{equation}\end{theorem}

\begin{proof}
If $k_3=0$ then (\ref{eq:thm1}) simplifies to $f-g=0$.  Putting $y_n=1+r_n$ and isolating the term containing $y_ny_{n+1}$, we find
\begin{eqnarray*}0&=&y_ny_{n+1}(y_n-q^n)^2(y_ny_{n-1}+k_1-1)=(k_1-1)(y_n-q^n)^2y_ny_{n-1}\\
&&+(q^n-y_n)(k_1-1)y_n\left[(y_n-q^n)+(q^n-1)(y_n-1+k_1)\right]\\
&&-q^n(y_n-1)(y_n-1+k_1)y_n\left[q^n(y_n-1)+(q^n-1)(k_1-1)\right].\end{eqnarray*}
With the substitution $y_n=\sqrt{1-k_1}x_n=px_n$ the first two terms can be completed to contain $(x_nx_{n-1}-1)(x_nx_{n+1}-1)$ and the equation simplifies to (\ref{qP5}).

Initial conditions for $a_n$ and $b_n$ (and, hence, $x_n$) can be computed by using the fact that the recurrence coefficients can be expressed in terms of the Hankel determinants containing the moments of the orthogonality measure \cite[Th. 4.2, p. 19; Ex. 3.1, p. 17]{Chihara}.  In particular, we will need that $b_0=\mu_1/\mu_0$ where $\mu_k$ is the $k$'th moment of the weight $w$.  Since $a_0=r_0=0$, we immediately find that $y_0=1$ and hence $x_0=1/p$.
As for $x_1$, we know from (\ref{eq733g}) that $r_1=-b_0R_0-k_1$.  Here, $R_0$ can be obtained from (\ref{eq734g}) with $n=0$ and $k_3=0$: we find that $R_0=q^{-1}k_2b_0$.  This leads to \[r_1=-k_1-\left(\frac{\mu_1}{\mu_0}\right)^2\frac{k_2}{q}\] and hence \[x_1=p-\frac{k_2}{pq}\left(\frac{\mu_1}{\mu_0}\right)^2.\]

The connection (\ref{aninThm2}) between $x_n$ and $a_n$ follows immediately from (\ref{eq737g}).  To obtain (\ref{bninThm2}), we use the squared (\ref{eq733g}) to write \[b_n^2=\frac{(r_{n+1}+r_n+k_1)^2}{R_n^2},\] where $R_n^2$ can be substituted using (\ref{Rn2g}).  Finally, substitute $r$ by $x$ using $x_n=(1+r_n)/p$ and eliminate $x_{n-1}$ using the $q$-Painlev\'{e} equation (\ref{qP5}).

\end{proof}

Clearly (in the case $k_3\neq 0$), we can also  derive a third order  difference equation for $r_{n}$ from
the system (\ref{eq733g})--(\ref{eq736g}) as follows. We can  get $R_{n+1}^2$ by either using relation (\ref{Rn2g}) with $n+1$ or  by squaring  (\ref{R{n+1}g}) and
using (\ref{Rn2g}). Hence, we can set them equal and in the result, we get a (cumbersome) expression involving only $r_{n-1},\,r_n,\,r_{n+1},\,r_{n+2}$.
%Clearly, by eliminating $R_n$ between  (\ref{Rn2-2}) and (\ref{eq7310-2}), we also get a second order second degree equation for $r_n$.

To find out which weights can give rise to a potential of the form
(\ref{pot gen}), we notice that it is sufficient if
$$\frac{w(x/q)}{w(x)}=Ax^2+Bx+C$$ for certain constants $A$, $B$,
$C$, since an easy calculation shows that in that case the potential
is given by (\ref{pot gen}) with $k_1=1-C$, $k_2=-Aq$ and $k_3=-Bq$.

If we define \begin{eqnarray*}v_1^\alpha(x)=x^\alpha,\quad v_2^c(x)&=&(cx;q)_\infty,\quad
v_3^c(x)=(cx^2;q^2)_\infty,\\
v_4^c(x)&=&(c/x;q)_\infty,\quad
v_5^c(x)=(c/x^2;q^2)_\infty,\end{eqnarray*} then %it is clear that
\begin{eqnarray*}
\frac{v_1^\alpha(x/q)}{v_1^\alpha(x)}=q^{-\alpha},\quad
\frac{v_2^c(x/q)}{v_2^c(x)}&=&\frac{q-cx}{q},\quad
\frac{v_3^c(x/q)}{v_3^c(x)}=\frac{q^2-cx^2}{q^2},\\
\frac{v_4^c(x/q)}{v_4^c(x)}&=&\frac{x}{x-c},\quad
\frac{v_5^c(x/q)}{v_5^c(x)}=\frac{x^2}{x^2-c}.\end{eqnarray*}
Hence it is clear which products of $v_i$ lead to a weight for which
the potential satisfies (\ref{pot gen}).  These include, among
others, the little $q$-Laguerre weight, the weight in
\cite[Sect.~7.3]{Lies}, products of rational functions and the weights above, the  weights in the following examples and others.

{\bf Example 1.}  In this example we consider the semi-classical little  $q$-Laguerre weight
\begin{equation}\label{weight}
w(x)=x^{\alpha} (qx;q)_{\infty} (c q x;q)_{\infty},\;\;\alpha>0,
\end{equation}
on the positive exponential lattice $\{ q^n\,|\, n\in\mathbb{N}_0
\}.$  The case
$c=-1$ was considered in \cite[Sect.~7.3]{Lies}.  The case $c=0$
gives the little $q$-Laguerre weight (and, hence, the recurrence coefficients are known explicitly).  We observe that
$w(0)=w(1/q)=0.$

The potential (\ref{potential}) is given by
$$u(x)=\frac{1}{1-q}\left(\frac{q}{x}-\frac{q^{1-\alpha}}{x}+q^{1-\alpha}(1+c)-c q^{1-\alpha}x\right)$$
and, hence, $k_1=1-q^{-\alpha},\;\;k_2=-cq^{1-\alpha},\;\;k_3=(1+c)q^{1-\alpha}$ in (\ref{pot gen}). We assume that $c\neq 0.$

Since $k_3=0$ if and only if $c=-1$, we get that in this case the variable $x_n=q^{\alpha/2}(r_n+1)$ satisfies
\begin{equation}\label{qP5part}
(x_n x_{n-1}-1)(x_n x_{n+1}-1)=\frac{q^{2n+\alpha}(x_n-q^{\alpha/2})^2(x_n-q^{-\alpha/2})^2}{(x_n-q^{n+\alpha/2})^2},
\end{equation}
 which is a particular case of $qP_V$ (\ref{qP5})
(with $\alpha=\beta=\gamma=\delta=q^{\alpha/2}$). This coincides with the result in \cite[Sect.~7.3]{Lies}. Note that equation (\ref{qP5part}) for the variable $r_n$ is given by
\begin{equation}\label{eq:rn}
r_n^2(r_n+1-q^{-\alpha})^2=q^{-2n}(r_n+1-q^n)^2((r_n+1) (r_{n+1}+1)-q^{-\alpha})((r_n+1) (r_{n-1}+1)-q^{-\alpha}).
\end{equation}

Next we study initial conditions for the recurrence coefficients of the weight (\ref{weight}) for general $c$.
 Using  (\ref{eq737g}), we have $$
a_n^2=\frac{1}{c} \,
q^{n-1+\alpha/2}\left(q^{n+\alpha/2}-x_n\right).
$$
To find an expression of $b_n$ in terms of $x_n$ we square the
expression of $b_n$ from (\ref{eq733g}) and substitute (\ref{Rn2g}).
We can also use (\ref{eq:rn}) to get rid of $x_{n-1}$. We get for
$n\geq 1$ $$ c b_n^2(x_n x_{n+1}
-1)=q^{2n}\left(1+q^{\alpha}-q^{\alpha/2}(x_n+x_{n+1})\right)^2.
$$ We  also get  that $$
b_0^2=\frac{1}{c}\left(q^{\alpha/2}x_1-1\right).
$$ Recalling that $b_0=\mu_1/\mu_0$, we immediately get that the initial values are given by
$x_0=q^{\alpha/2}$ (since $r_0=0$) and
\begin{equation}\label{x1}
x_1=q^{-\alpha/2}\left(1+c\frac{\mu_1^2}{\mu_0^2}\right),
\end{equation}
 where $\mu_k$ is the $k$-th moment of the weight (\ref{weight}). In fact, we can also calculate $\mu_k$ by definition
and get $$\mu_k=(1-q)(q;q)_{\infty}(cq;q)_{\infty} \,
{}_2\phi_1(0,0;cq;q;q^{\alpha+k+1}),$$ where the basic
hypergeometric function ${}_2\phi_1 $ is given by \cite[Sect. 0.2,
0.4]{KS}
$${}_2\phi_1(a_1,a_2;b_1;q;z)=\sum_{\ell=0}^{\infty}\frac{(a_1;q)_\ell (a_2;q)_\ell}{(b_1;q)_\ell}\frac{z^\ell}{(q;q)_\ell}.$$
Note that for $c=q^{\nu}$ the last expression (up to a factor) can
be written in terms of the modified $q$-Bessel function \cite{Rogov}%[Sect. 0.7]{KS}
$$I^{(1)}_{\nu}(z,q)=\frac{(q^{\nu+1};q)_{\infty}}{(q;q)_{\infty}}(z/2)^{\nu}{}_2\phi_1(0,0;q^{\nu+1};q;z^2/4)$$ with $z=2q^{(\alpha+k+1)/2}.$
Also note that a limiting case of Heine's transformation formula \cite[formula (0.6.9)]{KS} allows us to write this $q$-hypergeometric function as a ${}_0\phi_1$-function.

{\bf Example 2.}  In this example we consider another semi-classical generalization of the little $q$-Laguerre weight:
\begin{equation}\label{weight2}
w(x)=x^{\alpha} \frac{(qx;q)_{\infty} \left(\frac{c_1}{x};q\right)_\infty\left(\frac{qx}{c_1};q\right)_\infty}{\left(\frac{c_2}{x};q\right)_\infty},\;\;\alpha>0,\:c_1<0,\:c_2<0
\end{equation}
on the positive exponential lattice $\{ q^n\,|\, n\in\mathbb{N}_0
\}.$  The case where $c_1=c_2=1/c$ gives the weight from the previous example.  Again, it is clear that $w(0)=w(1/q)=0$.  It is easy to calculate that for this weight we get \[k_1=1-\frac{c_2}{c_1}q^{-\alpha},\qquad k_2=-\frac{q^{1-\alpha}}{c_1}\qquad{\rm and}\qquad k_3=\frac{c_2+1}{c_1}q^{1-\alpha}.\]

As mentioned earlier, to obtain a Painlev\'e equation, we need that $k_3=0$, hence $c_2=-1$.  So, following the outline given before, we see that \[x_n=\sqrt{-c_1}q^{\alpha/2}(r_n+1)\qquad{\rm and}\qquad a_n^2=-c_1q^{n+\alpha-1}(1-q^n+r_n)\] where $x_n$ satisfies (\ref{qP5}) with \[\alpha=\beta=\gamma=\delta=\sqrt{-c_1}q^{\frac{\alpha}{2}}.\]

%As in the previous example, starting from the squared (\ref{eq733g}) and substituting (\ref{Rn2g}), we obtain $b_n^2$ as a function of $r_{n-1}$, $r_n$, $r_{n+1}$.  The Painlev\'e equation can then be used to get rid of the $r_{n-1}$ dependency, so that we finally obtain \[b_n^2=...\]
As for the initial conditions, we find that \[x_0=\sqrt{-c_1}q^{\frac{\alpha}{2}}\] and \[x_1=\sqrt{-c_1}q^{\frac{\alpha}{2}}\left(-1+\frac{c_2}{c_1}q^{-\alpha}+\frac{b_0^2}{c_1q^\alpha}\right)+1.\]  Moreover, $b_0=\mu_1/\mu_0$, and it is easily seen that the $k$'th moment of this weight is given by \[\mu_k=(1-q)(q;q)_\infty\left(\frac{q}{c_1};q\right)_\infty\frac{(c_1;q)_\infty}{(c_2;q)_\infty}{}_2\phi_1\left(0,0;\frac{q}{c_2};q;\frac{c_1}{c_2}q^{\alpha+1+k}\right).\]  Hence $b_0$ can be written, up to a factor, as a fraction of two modified $q$-Bessel functions with $q^{-\nu}=c_2$.  Again, the $q$-hypergeometric function can alternatively be written as a ${}_0\phi_1$-function, using \cite[formula (0.6.9)]{KS}.

\section{Discussion}

In this paper we have shown that it is possible to study simultaneously recurrence coefficients in the three-term recurrence relation for  a large class of weights by using the technique of  ladder operators. The crucial point is to consider the potential (\ref{pot gen}) with parameters. This allows us to obtain a second degree second order discrete equation, which in some particular cases, can be further reduced to the  discrete Painlev\'e equation. It is an interesting open problem to try to classify the weights which lead to the appearance of the discrete Painlev\'e equations for the recurrence coefficients.

\section*{Acknowledgements}
GF is supported by MNiSzW Iuventus Plus Grant Nr 0124/IP3/2011/71. Part of this research was done when CS was visiting the University
of Warsaw, and he is grateful for their hospitality.

\end{document}